\documentclass{amsart}

\usepackage[T1]{fontenc}
\usepackage{fix-cm}

\usepackage[latin1]{inputenc} 
\usepackage{url}

\usepackage{amsmath}
\usepackage{amssymb}
\usepackage{amsthm}
\usepackage{mathrsfs}
\usepackage{bbm}
\usepackage{multirow, bigdelim}

\newtheorem{lem}[equation]{Lemma}
\newtheorem{thm}[equation]{Theorem}
\newtheorem{prop}[equation]{Proposition}
\newtheorem{cor}[equation]{Corollary}

\theoremstyle{definition}
\newtheorem{defn}[equation]{Definition}

\theoremstyle{remark}
\newtheorem{rem}[equation]{Remark}
\newtheorem{exm}[equation]{Example}

\DeclareMathOperator{\fix}{fix}
\DeclareMathOperator{\ran}{ran}
\DeclareMathOperator*{\essup}{ess\,sup}

\makeatletter

\newcommand{\Rmnum}[1]{\expandafter\@slowromancap\romannumeral #1@}
\makeatother

 \usepackage[usenames,dvipsnames]{color}   

\newcommand{\sr}[1]{{\text\sl r}\left(#1\right)}
\newcommand{\SG}[1][T]{\left(#1(t)\right)_{t\geq 0}}

\newcommand{\sse}{\subseteq}

\newcommand{\bk}{\backslash}
\newcommand{\la}{\lambda}

\newcommand{\A}{{\mathcal{A}}}
\newcommand{\M}{{\mathcal{M}}}

\newcommand{\Lop}[1][X]{{\mathcal L}\left(#1\right)}

\newcommand{\PtSp}[1]{P_\sigma\left(#1\right)}
\newcommand{\C}{\mathbb{C}}
\newcommand{\R}{\mathbb{R}}
\newcommand{\N}{\mathbb{N}}

\newcommand{\beq}{\begin{equation}}
\newcommand{\eeq}{\end{equation}}
\newcommand{\beqn}{\begin{equation*}}
\newcommand{\eeqn}{\end{equation*}}
\newcommand{\beqa}{\begin{eqnarray}}
\newcommand{\eeqa}{\end{eqnarray}}
\newcommand{\ba}{\begin{align}}
\newcommand{\ea}{\end{align}}
\newcommand{\ban}{\begin{align*}}
\newcommand{\ean}{\end{align*}}
\newcommand{\cl}[1][I]{\overline{#1}}
\newcommand{\lt}{\left}
\newcommand{\rt}{\right}

\author{Retha Heymann}
\address{Universit\"at T\"ubingen\\
Arbeitsbereich Funktionalanalysis\\
Mathematisch-Naturwissenschaftliche Fakult\"at\\
Auf der Morgenstelle 10\\
D-72076 T\"ubingen}
\email{rehe@fa.uni-tuebingen.de}

\title[Eigenvalues and Stability Properties of Multiplication Semigroups]{{\bf Eigenvalues and Stability Properties of Multiplication Operators and Multiplication Semigroups}}
\thanks{The author would like to express her sincere gratitude to Fatih Bayazit for many fruitful discussions. 
The author greatly appreciates the comments of Roland Derndinger, Klaus-Jochen Engel, D\'avid Kunszenti-Kov\'acs, Rainer Nagel and Pavel Zorin-Kranich.
The conversations with Hans Zwart and Felix Schwenninger have inspired some ideas of great value.}



\begin{document}

\begin{abstract}
We investigate uniform, strong, weak and almost weak stability of multiplication semigroups on Banach space valued $L^p$-spaces. We show that, under certain conditions, these properties can be characterized by analogous ones of the pointwise semigroups. Using techniques from selector theory, we prove a spectral mapping theorem for the point spectra of the pointwise and global semigroups and apply this as a major tool for determining almost weak stability.
\end{abstract}

\maketitle



One of the significant features of the Fourier transform is that it converts a differential operator into a multiplication operator induced by some scalar-valued function. The properties of the original operator are then determined by the values of this function. The same holds if a system of differential operators is transformed into a matrix valued multiplication operator on a vector valued function space. This motivates the systematic investigation of multiplication operators on Banach space valued function spaces.

Such operators (and semigroups generated by them) have been studied by, e.g. Holderrieth \cite{Hol91} as well as Hardt and Wagenf\"uhrer \cite{HaWa96} for matrix multiplication semigroups and by Arendt and Thomaschewski \cite{ThoAr05} and Graser \cite{Gra97} in the infinite dimensional case. See also \cite[Section 4]{MuNi11} and \cite{Neerven93}.
Qualitative properties of such semigroups, e.g. various stability concepts, are of great interest in control theory (see e.g. \cite{CIZ09}). In fact, motivated by these applications, Hans Zwart has proved a characterisation of strong stability of a multiplication semigroup in the finite dimensional case (see \cite{Zwart}), while so-called polynomial stability of multiplication semigroups is characterized in Theorem 4.4 of \cite{BEP06}.

The general question in this context is to what extent the global properties of a multiplication operator are determined by the local properties of the pointwise operators. In this paper we systematically investigate spectral and stability properties of multiplication semigroups. Our aim is thus to understand how these properties are related to those of the corresponding pointwise operators or semigroups, as explained below. As a major tool and also as a result of independent interest, we obtain a perfect characterisation of the eigenvalues of multiplication operators (Theorem \ref{thmPtSpM}). Furthermore, we indicate how the global and local stability properties are related for uniform (Theorem \ref{thmUniform}), strong (Theorem \ref{thmSG_STRONG}), weak (Proposition \ref{propWeak}) and almost weak (Theorem \ref{thmAwsSg}) convergence. Finally, we include a section in which we state analogous stability results for the powers of multiplication operators. 

Throughout this text we assume that the measure space $(\Omega,\Sigma, \mu)$ is $\sigma$-finite. For a separable Banach space $X$, $L^p(\Omega,X)$ denotes the Bochner space $L^p(\Omega,\Sigma,\mu;X)$ for a fixed value of $1\leq p \leq\infty$ {(see e.g \cite{abhn01}, or \cite{DU77}).}


\begin{defn}[Multiplication Operator, Pointwise Operator]\label{defMULT} 
Let $X$ be a separable Banach space and let $M:\Omega\rightarrow\Lop$ be
such that, for every $x\in X$, the function
$$
\Omega \ni s\mapsto M(s)x \in X
$$
is Bochner measurable. The {\em multiplication operator} $\M$ on $L^p(\Omega,X)$, with  $1\leq p \leq\infty$, is defined by
$$
(\M f)(s) := M(s)f(s) \text{ for all } s\in\Omega
$$
with
$$
D(\M)=\Bigl\{f\in L^p(\Omega,X)\ :\ M(\cdot)f(\cdot)\in L^p(\Omega,X) \Bigr\}.
$$
In this context, we call the operators $M(s)$ with $s\in\Omega$ the {\em pointwise operators} on $X$.
\end{defn}

\begin{rem}
The Bochner measurability of $s\mapsto M(s)x$ for all $x \in X$ implies that the function $M(\cdot)f(\cdot)$ is also measurable if $f\in L^p(\Omega,X)$, see \cite[Lemma 2.2.9]{Tho03}.
\end{rem}

\begin{defn}[Multiplication Semigroup]{\cite[Definition 2.3.8, p.\,45]{Tho03}}
If a $C_0$-semigroup $\bigl(\M(t)\bigr)_{t\geq 0}$ on $L^p(\Omega,X)$ consists of multiplication operators on $L^p(\Omega,X)$, it is called a {\em multiplication semigroup}. (For the general theory of $C_0$-semigroups, we refer to \cite{enna00}.)
\end{defn}


\begin{rem}\label{remNOT}
If $\bigl(\A,D(\A)\bigr)$ is the generator of the multiplication semigroup $\bigl(\M(t)\bigr)_{t\geq 0}$, then there exists a family of operators $A(s)$ with domain $D(A(s))$ on $X$ such that $A(s)$ is the generator of a $C_0$-semigroup on $X$ for all $s\in\Omega\bk N$ for some null set $N$ (see \cite[Theorem 2.3.15, p.\,49]{Tho03}). We call these semigroups on $X$ the {\em pointwise semigroups}.

We denote the multiplication semigroup $\bigl(\M(t)\bigr)_{t\geq 0}$ by $\left(e^{t\A}\right)_{t\geq 0}$ and the pointwise semigroups by $\bigl(e^{tA(s)}\bigr)_{t\geq 0}$ for all $s\in\Omega\bk N$.

Furthermore, for every $t\geq 0$, the function from $\Omega$ to $\Lop$, $s\mapsto e^{tA(s)}$, is measurable and the operator $e^{t\A}$ is the corresponding multiplication operator on $L^p(\Omega,X)$, cf. \cite[Proposition 2.3.12, p.\,48]{Tho03}.

\end{rem}


Before discussing stability properties, we recall the characterisation of boundedness of a multiplication operator by Thomaschewski (cf. \cite{Tho03}) or Klaus-J.\ Engel (cf. \cite[Chapter IX, Proposition 1.3]{en97}).



\begin{lem}\label{lemNORM}\cite[Proposition 2.2.14, p.\ 35]{Tho03}
Let $\M$ be a multiplication operator on $L^p(\Omega,X)$ induced by the function $M(\cdot)$ as in Definition \ref{defMULT}. The operator $\M$ is bounded if and only if the function $M(\cdot)$ is essentially bounded{, i.e. if the function $M(\cdot)$ is bounded up to a set of a measure zero. In this case}, we have that
\begin{align*}
\|\M\| &= \essup_{s\in\Omega}{\|M(s)\|} \\
  &:= \inf\Bigl\{ C \geq 0\ :\ \mu\bigl(\left\{s\in \Omega\ :\ \|M(s)\| > C\right\}\bigr) = 0 \Bigr\}.\\
\end{align*}

\end{lem}

We are interested in the extent to which stability properties of the pointwise semigroups determine stability properties (see the sections below) of the corresponding multiplication semigroup. This is not always the case as is illustrated by a simple modification of Zabzyk's classical counterexample to the spectral mapping theorem for $C_0$-semigroups (cf. \cite{enna00} p. 273, Counterexample 3.4).

Other examples are discussed in \cite[Section 2]{CIZ09}.

\section{Eigenvalues of a multiplication operator}

Since many stability properties can be characterised by spectral properties (see, for instance, \cite[Chapter IV and V]{enna00}), we start by investigating the relation between the pointwise spectra $\sigma(M(s))$ with $s\in\Omega$ and the global spectrum $\sigma(\M)$.

As a first and important step, we characterise the eigenvalues of $\M$ via $M(s)$ with $s\in\Omega$.

We take note of the fact that the nontrivial implication in the following theorem is essentially a selector result. See, for instance, \cite{KurRyll_65} for a standard selector theorem. It is possible to present a proof based on methods from this theory, but our proof is more elementary and not less elegant.

We would like to mention that the idea of the following proof has been sparked by discussions with Hans Zwart.

\begin{thm}\label{thmPtSpM}
For a multiplication operator $\M$ on $L^p\left(\Omega,X\right)$ with $1\leq p\leq\infty$ induced by the pointwise operators $\{M(s)\ :\ s\in\Omega\}$ and an arbitrary $\lambda \in \C$, the following equivalence holds:
$$
\la \in \PtSp{\M} \Longleftrightarrow \la \in \PtSp{M(s)} \mbox{ for $s \in Z$, }
$$
where $Z$ is a measurable subset of $\Omega$ with $\mu(Z)>0$.

\end{thm}


\begin{proof}
Assume that the pointwise operators $M(s)$ are injective, i.e., the kernels of $M(s)$ are zero for almost all $s\in\Omega$. Choose an arbitrary function $0\not=f\in L^p(\Omega,X)$. It is clear that $\M f \not= 0$, hence the kernel of $\M$ is zero.

Assume now that $M(s)$ is not injective for all $s$ in a set $N$ such that $\mu(N) > 0$. Without loss of generality, take $N = \Omega$.

\begin{description}
  \item[{Idea of proof}] \hfill \\
We will construct a sequence of countably-valued measurable functions $(f_m)$ in $L^p(\Omega,X)$, converging pointwise to some non-zero function such that the sequence of functions $(\M f_m)$ converges to zero in norm.

Choose 
 $$
   W := \{w_n : n\in\N\}
 $$
as a countable dense subset of the unit sphere of $X$. 

We construct for every $m\in\N$ countably many disjoint, measurable sets 
  $\tilde\Omega_{(n_1,n_2,\hdots,n_m)}$
indexed by $(n_1,n_2,\hdots,n_m) \in \prod_{k=1}^m\N$ such that
   \beq \label{eq:normM}
     \|M(s)w_{n_m}\| \leq \frac{1}{m} \text{ for all } s \in \tilde\Omega_{(n_1,n_2,\hdots,n_m)},
   \eeq

   \begin{align}\label{eq:inclusion}
     \Omega = \bigcup_{n_1\in\N} \bigcup_{n_2\in\N}\cdots \bigcup_{n_m\in\N} \tilde\Omega_{(n_1,n_2,\hdots,n_m)},
   \end{align}

and   
   \beq \label{eq:subset}
     \tilde\Omega_{(n_1,n_2,\hdots,n_m)} \sse \tilde\Omega_{(n_1,n_2,\hdots,n_{m-1})}.
   \eeq
Furthermore, if $s \in \tilde\Omega_{(n_1,n_2,\hdots,n_m)}$ then, for every $m \leq r \in \N$, there exists a convergent sequence $(a_k)_{k\in\N} \sse W$ with $a_k = w_{n_k}$ for $1 \leq k \leq m$, such that $\| M(s)a_k \| \leq \frac{1}{k}$ for $1 \leq k \leq r$.

Then, for every $m\in\N$, we define $f_m$ as 
 $$
   f_m := \sum_{(n_1,n_2,\hdots,n_{m}) \in \prod_{k=1}^{m}\N} {\mathbbm{1}_{\tilde\Omega_{(n_1,n_2,\hdots,n_m)}} w_{n_m}}.
 $$
For every $s\in\Omega$, the sequence
 $$
   (f_m(s))_{m\in\N} = (w_{n_m})_{m\in\N}
 $$
converges, $\|f_m(s)\| = \| w_{n_m} \| = 1 $ for all $m\in\N$ and
 $$
   \| M(s)f_m(s) \| = \| M(s)w_{n_m} \| \leq \frac{1}{m}.
 $$

  \item[{Definitions}] \hfill \\

Consider the following set of (convergent) sequences
 $$
   \mathcal{W} := 
       \left\{
         \begin{array}{ll}
	          (a_{k})_{k\in\N} : & a_{k} \in W \text{ for all } k\in\N, \text{ and } \\ 
	                             & \text{ for each } q\in\N , \|a_q - a_{q+j}\| < \frac{1}{q} \\
	                             & \text{ for all } j\in\N
         \end{array}
       \right\}.
 $$
Note that, to each sequence $(a_{k})_{k\in\N} \in \mathcal{W}$, there is a corresponding function 
  \beq \label{eq:beta}
    \beta : \N \rightarrow \N 
  \eeq
such that $a_k = w_{\beta(k)}$ for all $k\in\N$. Hence we can write $(a_k) = (w_{\beta(k)})$.

For each $m\in\N$ define the set $\mathcal{W}_{(n_1,n_2,\ldots,n_m)} \sse \mathcal{W}$ for each $m$-tuple $(n_1,n_2,\ldots,n_m) \in \prod_{j=1}^{m}\N$ as
 $$
   \mathcal{W}_{(n_1,n_2,\ldots,n_m)} := 
     \left\{
       \begin{array}{ll}
	          (a_{k})_{k\in\N} \in \mathcal{W} : & a_{k} = w_{n_k} \text{ for } 1 \leq k \leq m
        \end{array}
     \right\}.
 $$ 

Now, for all $n,j \in \N$, define the subsets
	  $$
	    \Omega_{n,j} := \left\{s \in \Omega : \|M(s)w_n\| \leq \frac{1}{j}\right\}
	  $$
of $\Omega$.


  \item[{Construction}] \hfill \\

  Take $m=1$.
  We now define the sets $\tilde\Omega_{(n_1)}$ for each $n_1 \in \N$ as
  
    \begin{align}
        \Omega_{(n_1)} 
        &:=
          \left\{
             \begin{array}{ll}
                s \in \Omega\ : & \text{ for each } r\in\N \text{ there exists a sequence } \\
                                & (a_{k})_{k\in\N} = (w_{\beta(k)})_{k\in\N} \in \mathcal{W}_{(n_1)} \\ 
                                & \text{ such that } s\in\bigcap_{j=1}^r{\Omega_{\beta(j),j}} 
	           \end{array}
          \right\}  \notag\\                       
       &\ = 
           \bigcap_{r \in \N}
           \left(
             \bigcup_{ (w_{\beta(k)})_{k\in\N} \in \mathcal{W}_{(n_1)}
                     }
                \left( \bigcap_{j=1}^r{\Omega_{\beta(j),j}}
                \right)
           \right)  \label{eq:cntble}
  \intertext{and}
      \tilde\Omega_{(n_1)} &:= \Omega_{(n_1)} \bk \bigcup_{q < n_1}\tilde\Omega_{(q)}.\notag
    \end{align}

The set $ \Omega_{(n_1)}$ is the countable intersection of the union of certain measurable sets of the form $\bigcap_{j=1}^r{\Omega_{\beta(j),j}}$. As we see in \eqref{eq:cntble} above, this union consists of those sets, for each of which the corresponding sequence $ (w_{\beta(j)})_{j\in\N}$ is in $\mathcal{W}_{(n_1,n_2,\ldots,n_m)} $ which is an uncountable set. However, there are only countably many sets of the form $\bigcap_{j=1}^r{\Omega_{\beta(j),j}}$ for a fixed $r\in\N$ and hence the union can be written as a countable union. Hence $\Omega_{(n_1)}$ and also $\tilde\Omega_{(n_1)}$ are measurable.
It now holds that
  \beqn
     \|M(s)w_{n_1}\| \leq 1
   \eeqn
for all $s \in \tilde\Omega_{(n_1)}$ and
   \begin{align*}
     \Omega = \bigcup_{n_1\in\N} \tilde\Omega_{(n_1)}.
   \end{align*}

Furthermore, if $s \in \tilde\Omega_{(n_1)}$ then, for every $r\in\N$, there exists a (convergent) sequence $(a_k)_{k\in\N} \in \mathcal{W}$ 
with $a_1 = w_{n_1}$ and $\| M(s)a_k \| \leq \frac{1}{k}$ for $1 \leq k \leq r$.


%

Now define the function
  $$
    f_1 := \sum_{n_1 \in \N}\mathbbm{1}_{\tilde\Omega_{(n_1)}} w_{n_1}.
  $$
Observe that $\|f_1\| = 1$ and $\|\M f_1\| \leq 1$ as desired.


\vspace*{5pt}
The recursive definition of the sets $\tilde\Omega_{(n_1,n_2,\hdots,n_m)}$ now follows.

Let $m \geq 2$ and assume that we have disjoint sets $\tilde\Omega_{(n_1,n_2,\hdots,n_{m-1})}$ with $(n_1,n_2,\hdots,n_{m-1}) \in \prod_{k=1}^{m-1}\N$ such that
   \beqn
     \|M(s)w_{n_{m-1}}\| \leq \frac{1}{m-1} \text{ for all } s \in \tilde\Omega_{(n_1,n_2,\hdots,n_{m-1})},
   \eeqn

   \begin{align*}
     \Omega = \bigcup_{n_1\in\N} \bigcup_{n_2\in\N}\cdots \bigcup_{n_{m-1}\in\N} \tilde\Omega_{(n_1,n_2,\hdots,n_{m-1})},
   \end{align*}
and   
    \beqn
     \tilde\Omega_{(n_1,n_2,\hdots,n_{m-1})} \sse \tilde\Omega_{(n_1,n_2,\hdots,n_{m-2})}.
   \eeqn
   
Moreover, for every $s \in \tilde\Omega_{(n_1,n_2,\hdots,n_{m-1})}$ and every $m-1 \leq r \in \N$, there exists a sequence $(a_k) \in \mathcal{W}$ with $a_k = w_{n_k}$ for $1 \leq k \leq m-1$ and $\| M(s)a_k \| \leq \frac{1}{k}$ for $1 \leq k \leq r$.   

Then, for every $(n_1,n_2,\hdots,n_{m-1}) \in \prod_{k=1}^{m-1}\N$ 
  and $n_m \in \N$, define the measurable sets

    \begin{align*}
        \Omega_{(n_1,n_2,\hdots,n_m)} 
        &:=
          \left\{
               \begin{array}{ll}
	                 s \in \Omega\ : &\text{ for all } r\in\N \text{ there exists a sequence} \\ 
                                   & (a_k)_{k\in\N} = (w_{\beta(k)})_{k\in\N}\in \mathcal{W}_{(n_1,n_2,\hdots,n_m)}  \\ 
                                   &  \text{ such that } s\in\bigcap_{j=1}^r{\Omega_{\beta(j),j}} 
               \end{array}
           \right\}\\
  \intertext{and}
      \tilde\Omega_{(n_1,n_2,\hdots,n_m)} &:=  \tilde\Omega_{n_1,n_2,\hdots,n_{m-1}} \cap \left( \Omega_{(n_1,n_2,\hdots,n_m)} \bk \bigcup_{q < n_m}\tilde\Omega_{(n_1,n_2,\hdots,n_{m-1},q)}\right).
    \end{align*}
For every $\tilde\Omega_{(n_1,n_2,\hdots,n_m)}$, the properties \eqref{eq:normM}, \eqref{eq:inclusion} and \eqref{eq:subset} hold and if $s \in \tilde\Omega_{(n_1,n_2,\hdots,n_m)}$ then, for every $m \leq r \in \N$, there exists a sequence $(a_k) \in \mathcal{W}$ with $a_k = w_{n_k}$ for $1 \leq k \leq m$ and $\| M(s)a_k \| \leq \frac{1}{k}$ for $1 \leq k \leq r$.

        
       
Define the function
  $$
   f_m := \sum_{(n_1,n_2,\hdots,n_m) \in \prod_{k=1}^m\N}{\mathbbm{1}_{\tilde\Omega_{(n_1,n_2,\hdots,n_m)}} w_{n_m}}.
  $$

Thus the sequence $(f_m)_{m\in\N}$ has been constructed with the desired properties.

Indeed, for every $m\in\N$, we have $\|f_m\| \geq 1$, because of \eqref{eq:inclusion} and the fact that, if $s\in \tilde\Omega_{(n_1,n_2,\hdots,n_m)}$, then $\|f(s)\| = \|w_{n_m}\| = 1$. We also have that $\|\M f_m\| \leq \frac{1}{m}$. Furthermore, the sequences $(f_m(s))_{m\in\N}$ are convergent for every $s\in\Omega$. The pointwise limit $f$ is measurable, nonzero and in the kernel of $\M$.

\end{description} 

\end{proof}


We obtain the following corollary directly from the above theorem by using the spectral mapping theorem for the point spectrum of the resolvent of a closed operator (see, for instance, \cite[Chapter IV, Theorem 1.13]{enna00}).


\begin{cor}\label{corPtSpA}
If $\left(e^{t\A}\right)_{t\geq 0}$ is a multiplication semigroup with generator $\A$ and $\lambda\in\C$, then the following equivalence holds.

$$
\la \in \PtSp{\A} \Longleftrightarrow \la \in \PtSp{A(s)} \mbox{ for $s \in Z$, }
$$
where $Z$ is a measurable subset of $\Omega$ with $\mu(Z)>0$.

\end{cor}


\section{Uniform Stability}

This short section is devoted to the strongest notion of stability, namely uniform stability.


\begin{defn}\cite[Definition V.1.1, p. 296]{enna00}
A $C_0$-semigroup $\SG$ of operators on a Banach space is called {\em uniformly stable} if $\|T(t)\|\rightarrow 0$ as $t\rightarrow\infty$.
\end{defn}

Lemma \ref{lemNORM} immediately leads to the following characterisation of uniform stability of multiplication semigroups.


\begin{thm} \label{thmUniform} 
Let $\bigl(e^{t\A}\bigr)_{t\geq 0}$ be a multiplication semigroup on $L^p(\Omega,X)$  with $1\leq p \leq \infty$.
 Then the following are equivalent.
\begin{enumerate}
	\item[(a)] The semigroup $\bigl(e^{t\A}\bigr)_{t\geq 0}$ is uniformly stable.
	\item[(b)] The pointwise semigroups converge to $0$ in norm, uniformly in $s$, i.e.
    $$
      \essup_{s\in\Omega}{\left\|e^{t A(s)}\right\|} \rightarrow 0\ \text{ as }\ t\rightarrow\infty.
    $$ 
	\item[(c)] At some $t_0 > 0$, the spectral radii $\sr{e^{t_0 A(s)}}$ of the pointwise semigroups 
	           satisfy $\essup_{s\in\Omega}{\sr{e^{t_0 A(s)}}} < 1$.
	\item[(d)] There exist constants $C \geq 1$ and $\epsilon > 0$ such that $\|e^{tA(s)}\| \leq Me^{-t\epsilon }$ for all $t\geq 0$ and almost all $s$.
\end{enumerate}
\end{thm}


\begin{rem}
Note that, due to Remark 6, the uniform stability of a multiplication semigroup is independent of the value of $p$ as long as $1 \leq p \leq \infty$.
\end{rem}

\section{Strong Stability}

In this section we study strong stability in our context.


\begin{defn}\cite[Definition V.1.1, p. 296]{enna00}
A $C_0$-semigroup $\SG$ of operators on a Banach space is called {\em strongly stable} if $\|T(t)x\|\rightarrow 0$ as $t\rightarrow\infty$ for all $x \in X$.
\end{defn}


We now show that a multiplication semigroup is strongly stable if and only if the pointwise semigroups are uniformly bounded in $s$ and strongly stable almost everywhere. The backward implication was proved by Curtain-Iftime-Zwart in \cite{CIZ09} for the special case where $\Omega=[0,1],\ p=2$ and $X= \C^n$. The other implication was conjectured in the same paper and has since been proved by Hans Zwart (see Theorem on p. 3 in \cite{Zwart}), again for $X=\C^n$. 


\begin{thm} \label{thmSG_STRONG}
Suppose that the multiplication operator $\A$ generates a $C_0$-semigroup of multiplication operators $(e^{t\A})_{t \geq 0}$ on $L^p(\Omega,X)$ with $1 \leq p < \infty$ such that $\|e^{t\A}\| \leq C$ for all $t\geq 0$ and some constant $C >0$.  
Then the following are equivalent.
\begin{enumerate}
 \item[(a)] The semigroup $\lt(e^{t\A}\rt)_{t\geq 0}$ is strongly stable.
 \item[(b)] The pointwise semigroups $\lt(e^{tA(s)}\rt)_{t\geq 0}$ on $X$ are strongly stable for almost all $s\in \Omega$.
\end{enumerate}

\end{thm}


\begin{proof}
Note that, by Lemma \ref{lemNORM}, we have $\|e^{t\A}\| \leq C$ for all $t\geq 0$ if and only if $\|e^{tA(s)}\| \leq C$ for almost all $s\in \Omega$ and for all $t\geq 0$.

 (a)\,$\implies$\,(b): Assume that $\|e^{t\A}f\|\rightarrow 0$ as $t\rightarrow\infty$ for all $f\in L^p(\Omega,X)$.

Choose an arbitrary $x\in X$ and define the function $f_x:\Omega\rightarrow X$ by
$$
f_x(s):= x
$$
for all $s\in\Omega$.
Since $\Omega$ is $\sigma$-finite, we can write $\Omega=\cup_{n\in\N}{\Omega_n}$ where $\mu(\Omega_n) < \infty$ for every $n\in\N$. Then $\mathbbm{1}_{\Omega_n}{f_x} \in L^p(\Omega,X)$ for every $n\in\N$.
By assumption, we have that
$$
\left\| e^{t\A} \bigl( \mathbbm{1}_{\Omega_n}{f_x} \bigr) \right\|\rightarrow 0 \quad \text{as} \quad t\rightarrow\infty
$$
for every $n\in\N$.
Therefore, the Riesz subsequence theorem (see e.g. proof of \cite[Chapter I, Theorem 3.12]{RudAna}) implies that, for every $n\in\N$, there exists a sequence $(t_k)_{k \in \N} \subset \R_+$ tending to infinity as $k\rightarrow\infty$, such that the functions $e^{t_k\A}\bigl( \mathbbm{1}_{\Omega_n}{f_x} \bigr)$ converge pointwise to $0$, almost everywhere, i.e.
\begin{align}\label{cvg}
\left\|e^{t_kA(s)} \bigl( \mathbbm{1}_{\Omega_n}(s){f_x}(s) \bigr) \right\|_{ X} = \left\|e^{t_kA(s)}x\right\|_{ X} \rightarrow 0 \quad \text{as} \quad k\rightarrow\infty
\end{align}
for $s\in  \Omega_n\bk N_{(x,n)}$, where $N_{(x,n)}$ is a subset of $\Omega$ of measure $0$. We now show that \eqref{cvg} implies that $\left\|e^{tA(s)}x\right\|_{ X} \rightarrow 0$ as $t\rightarrow\infty$ for $s\in \Omega_n\bk N_{(x,n)}$.

Let $\epsilon>0$. Then $\left\|e^{t_kA(s)}x\right\|_{ X}<\frac{1}{C}\epsilon$ for all $k$ greater than or equal to some $k_\epsilon\in\N$. For each $t>t_{k_\epsilon}$, we can write $t=t_{k_\epsilon}+r$ where $r\in\R_+$. Then we have that
\begin{align*}
\left\|e^{tA(s)}x\right\|_{ X} &=    \left\|e^{(r + t_{k_\epsilon})A(s)}x\right\|_{ X} \\
                      &=    \left\| \lt(e^{rA(s)}\rt) \lt(e^{t_{k_\epsilon}A(s)}x\rt) \right\|_{ X} \\
                      &\leq \left\|e^{rA(s)}\right\| \left\|e^{t_{k_\epsilon}A(s)}x\right\|_{ X} \\
                      &\leq C  \left\|e^{t_{k_\epsilon}A(s)}x\right\|_{ X} \\
                      &\leq C  \frac{1}{C}\epsilon \\
                      &=    \epsilon.
\end{align*}
Hence,
\begin{align}\label{cvg2}
 \left\|e^{tA(s)}x\right\|_{ X}\rightarrow 0 \quad \text{as} \quad t\rightarrow\infty
\end{align}
for all $s\in \Omega_n\bk N_{(x,n)}$. It follows that \eqref{cvg2} holds for all $s\in\Omega\bk{\left(\cup_{n\in\N}N_{(x,n)} \right)}$. Observe that $N_x := \cup_{n\in\N}N_{(x,n)}$, being a countable union of null sets, is also a null set and hence we have the convergence \eqref{cvg2} almost everywhere.

For each $x \in  X$, 
\eqref{cvg2} holds for all $s\in \Omega\bk N_x$. Now, choose any countable dense subset $C\subset X$. Then \eqref{cvg2} holds for each $x\in C$, for all $s\in \Omega\bk\left(\cup_{x\in{C}}{N_x}\right)$. Note that $\cup_{x\in{C}}{N_x}$ is also null set. So we have that \eqref{cvg2} holds for all $x$ in a dense subset of $ X$, almost everywhere. Because the semigroups $\lt(e^{tA(s)}\rt)_{t\geq 0}$ are bounded, it follows that \eqref{cvg2} holds for all $x\in X$, almost everywhere, which is what we wanted to prove.

(b)\,$\implies$\,(a): Assume that $\left\|e^{tA(s)}x\right\|_{ X}\rightarrow 0$ as $t\rightarrow\infty$ for all $x\in X$ and almost all $s\in \Omega$.

Choose an arbitrary function $f\in L^p(\Omega, X)$. Then
$$
\left\|e^{tA(s)}f(s)\right\|_{ X}^p \leq C^p\left\|f(s)\right\|_{ X}^p
$$
for almost all $s\in \Omega$.
Hence, the functions $\left\|e^{t\A}f(\cdot)\right\|_{ X}^p$ are dominated by the integrable function $\left\|Mf(\cdot)\right\|_{ X}^p$. Because of our assumption, we know that\\ $\left\|e^{tA(s)}f(s)\right\|_{ X}^p\rightarrow 0$ as $t\rightarrow\infty$ for almost all $s\in \Omega$. It now follows from Lebesgue's dominated convergence theorem that
$$\int_{\Omega}{\left\|e^{tA(s)}f(s)\right\|_{ X}^p{\text d}s}=\left\|e^{t\A}f\right\|^p\rightarrow 0\quad\text{ as }\quad t\rightarrow\infty.$$

Thus the proof is complete.

\end{proof}


\begin{rem}
As before, strong stability of a multiplication semigroup is independent of the value $p$, as long as $1\leq p < \infty$.
\end{rem}

\section{Weak Stability}
The concept of weak stability is the most difficult to investigate in this context. We include a trivial example in the scalar case, where the multiplication semigroup is weakly stable, but where none of the pointwise semigroups are.  


\begin{exm}\label{exm1}
We use the notation introduced in Remark \ref{remNOT} with $\Omega=[0,1]$, $X=\C$ and $A(s):=is$ for all $s\in [0,1]$. Then the semigroup $\left(e^{t\A}\right)_{t\geq 0}$ is weakly stable, but none of the pointwise semigroups are. 
\end{exm}


We show in the proposition below that the weak stability of almost every pointwise semigroup $\left(e^{tA(s)}\right)_{t\geq 0}$ does imply that the multiplication semigroup $\left(e^{t\A}\right)_{t\geq 0}$ is weakly stable.


\begin{prop} \label{propWeak}
Let $\left(e^{t\A}\right)_{t\geq 0}$ be a bounded multiplication semigroup on the space $L^p(\Omega,X)$, with $1\leq p<\infty$, where $X$ is a reflexive Banach space. 

If the pointwise semigroups $\left(e^{tA(s)}\right)_{t\geq 0}$ are weakly stable for almost all $s\in\Omega$, then the multiplication semigroup $\left(e^{t\A}\right)_{t\geq 0}$ is weakly stable.
\end{prop}


\begin{proof}
Assume that the pointwise semigroups $\left(e^{t A(s)}\right)_{t\geq 0}$ are weakly stable for almost all $s\in\Omega$. Then there exists a null set $N\subset\Omega$ such that
  $$
    \psi\bigl(e^{tA(s)}x\bigr) \rightarrow 0 \text{ as } t\rightarrow\infty
  $$
for all $x\in X$, $\psi\in X'$, and for all $s\in\Omega\bk N$, where $X'$ denotes the continuous dual space of $X$. Since $X$ is a reflexive Banach space, the dual space of $L^p(\Omega,X)$ is $L^q(\Omega,X')$, where $\frac{1}{p} + \frac{1}{q} = 1$ (see \cite[Theorem 8.20.5, p.\,607]{Don95}). Choose arbitrary functions $f\in L^p(\Omega,X)$ and $g\in L^q(\Omega,X')$ (or $g\in L^\infty(\Omega,X')$, if $p=1$). Then
  \begin{align*}
    \left| g(s)  \biggl( e^{tA(s)}f(s) \biggr)  \right|
                                   & \leq \|g(s)\|_{X'} \|e^{t A(s)}f(s)\|_X \\
                                   & \leq C \|g(s)\|_{X'} \|f(s)\|_X\\
  \end{align*}
almost everywhere.

Since the real valued function $\|g(\cdot)\|_{X'}$ is in $L^q$ (or in $L^\infty$) and the function $\|f(\cdot)\|_X$ is in $L^p$, we have that the function $C \|g(\cdot)\| \|f(\cdot)\|_X$ is in $L^1$ and hence integrable. The functions $ g(\cdot)\Bigl(e^{tA(\cdot)}f(\cdot)\Bigr) $ are integrable, converge pointwise almost everywhere to $0$ as $t\rightarrow \infty$ and are dominated by the function $C \|g(\cdot)\|_{X'} \|f(\cdot)\|_X$. It follows from Lebesgue's Dominated Convergence Theorem that
$$
\int_\Omega{ g(s)\bigl(e^{tA(s)}f(s)\bigr) \text{d}}\mu(s)\rightarrow 0\text{\ \ as\ \ }t\rightarrow\infty.
$$
Hence, the semigroup $\left(e^{t\A}\right)_{t\geq 0}$ is weakly stable.
\end{proof}


\begin{rem}
This result is independent of the value of $p$, as long as $1\leq p<\infty$.
\end{rem}

\section{Almost Weak Stability}

We now investigate the almost weak stability of the multiplication semigroup $\bigl(e^{t\A}\bigr)_{t\geq 0}$ and of the pointwise semigroups $\bigl(e^{t A(s)}\bigr)_{t\geq 0}$ with $s\in\Omega$. In order to define this stability concept, we mention that the density of a Lebesgue measurable subset $Z$ of $\R_+$ is $d := \lim_{t\rightarrow\infty}{\frac{\mu(Z\cap[0,t])}{t}}$ ($\mu$ being the Lebesgue measure), whenever the limit exists.

%
%
%
%

\begin{defn}[Almost Weak Stability]
Let $\SG$ be a $C_0$-semigroup on a reflexive Banach space $X$. Then $\SG$ is  called {\em almost weakly stable} if there exists a Lebesgue measurable set $Z\subset \R_+$ of density $1$ such that
$$
T(t) \rightarrow 0 \quad \text{as} \quad t \to \infty \,,\ t\in Z,
$$
in the weak operator topology.
\end{defn}


The main result of this section is a characterisation of almost weak stability of a multiplication semigroup via Ces\`aro stability (see Definition \ref{defCesaroStable}) of the pointwise semigroups. This is remarkable since it is not true that almost weak stability of a multiplication semigroup and that of the pointwise semigroups are equivalent. We introduce Ces\`aro stability which seems to be the appropriate concept.


\begin{defn}[Ces\`aro mean; Mean Ergodic Semigroup, Mean Ergodic Projection](\cite[p.\,20, Chapter I, Definition 2.18, 2.20]{eis10})
Let $\SG$ be a $C_0$-semigroup of operators on a Banach space $X$ with generator $\bigl(A,D(A)\bigr)$. For every $t > 0$, the {\em Ces\`aro mean} $S(t)\in\Lop[X]$ is defined by
$$
S(t)x := \frac{1}{t}\int_{0}^t{T(s)x\text{d}s}
$$
for all $x\in X$.
The semigroup $\SG$ is called {\em mean ergodic} if the Ces\`aro means converge pointwise as $t$ tends to $\infty$. In this case the operator $P \in \Lop[X]$ defined by
$$
Px := \lim_{t\to \infty}{S(t)x}
$$
is called the {\em mean ergodic projection} corresponding to $\SG$.
\end{defn}


\begin{rem}\cite[p.\,21-22; Chapter I; Remark 2.21, Proposition 2.24 and Theorem 2.25]{eis10}
The mean ergodic projection $P$ commutes with  $\SG$ and is indeed a projection. A bounded $\SG$ is mean ergodic if and only if $X = \ker{A} \oplus \cl[\ran{A}]$, where $\ker{A}$ and $\ran{A}$ denote, respectively, the kernel and range of $A$. One also has that $\ran{P}=\ker{A}=\fix{\SG}$ and $\ker{P}=\cl[\ran{A}]$ where $\fix{\SG}$ is the fixed space of $\SG$.
\end{rem}


\begin{defn}[Ces\`aro Stability of Semigroups]\label{defCesaroStable}
A semigroup $\SG$ is called {\em Ces\`aro stable} if the Ces\`aro means converge to 0 as $t\rightarrow\infty$, i.e. the mean ergodic projection is the $0$ operator.
\end{defn}


\begin{thm}\label{thmAwsSg}
Let $X$ be a reflexive, separable Banach space.
Take $(\A,D(\A))$ to be the generator of a bounded multiplication semigroup $\bigl(e^{t\A}\bigr)_{t\geq 0}$ on $L^p(\Omega,X)$ with $1<p<\infty$, and let $A(s)$ with $s\in\Omega$ be the family of operators corresponding to $\A$.

The semigroup $\bigl(e^{t\A}\bigr)_{t\geq 0}$ is almost weakly stable if and only if, for each $ir \in i\R$, the rescaled pointwise semigroups $\bigl(e^{irt}e^{tA(s)}\bigr)_{t\geq 0}$ are Ces\`aro stable, almost everywhere.

%

\end{thm}


\begin{rem}
If the pointwise semigroups are almost weakly stable, almost everywhere, then the rescaled pointwise semigroups $\bigl(e^{irt}e^{tA(s)}\bigr)_{t\geq 0}$ are Ces\`aro stable for each $ir \in i\R$ almost everywhere which implies, by the above theorem, that the multiplication semigroup $\bigl(e^{t\A}\bigr)_{t\geq 0}$ is almost weakly stable.
\end{rem}


In order to prove Theorem \ref{thmAwsSg} we need the following characterisation of almost weak stability of a semigroup on a reflexive Banach space via the point spectrum of its generator.


\begin{lem}{\cite[Chapter II, Theorem 4.1]{eis10}}\label{lemAwsSg}
Let $\bigl(T(t)\bigr)_{t\geq 0}$ be a bounded $C_0$-semigroup with generator $(A, D(A))$ on a reflexive, separable Banach space $X$. Then the following are equivalent.
\begin{enumerate}
	\item $\bigl(T(t)\bigr)_{t\geq 0}$ is almost weakly stable.
	\item $\PtSp{A}\cap i\R = \emptyset$, where $\PtSp{A}$ is the point spectrum of $A$.
\end{enumerate}
\end{lem}

We are now ready to prove Theorem \ref{thmAwsSg} by using the above characterisation as well as the spectral mapping result of Theorem \ref{thmPtSpM}.


\begin{proof}
Using Lemma \ref{lemAwsSg} almost weak stability of $\bigl(e^{t\A}\bigr)_{t\geq 0}$ is equivalent to $\PtSp{\A}\cap i\R = \emptyset$. By Corollary \ref{corPtSpA}, this is equivalent to the following: for each $ir\in i\R$, $\mu\bigl(\{ s\in\Omega | ir \in \PtSp{A(s)} \}\bigr) = 0$. This in turn is equivalent to the fact that, for each $ir\in i\R$, the rescaled pointwise semigroups $\bigl(e^{irt}e^{tA(s)}\bigr)_{t\geq 0}$ are Ces\`aro stable almost everywhere.


\end{proof}

%
%
%

\begin{exm}
Once again, we consider the multiplication semigroup of Example \ref{exm1}. In this example we have for the point spectrum that $\PtSp{A(s)} \cap i\R \not=\emptyset$ for each $s\in\Omega$, since $is\in\PtSp{A(s)}$ for all $s\in[0,1]$. It follows that none of the pointwise semigroups are almost weakly stable, whereas the point spectrum of $\A$ is empty, which means that the semigroup $\bigl(e^{t\A}\bigr)_{t\geq 0}$ is almost weakly stable.
\end{exm}


\begin{rem}
Since the stability is determined by the pointwise semigroups, the value of $p$ is irrevelant, as long as $1<p<\infty$. 
\end{rem}

\section{Stability of Multiplication Operators}
It is possible to develop analogous results for time-discrete semigroups of the form $\{\M^n \mid \ n \in \mathbb{N}\}$ for a bounded multiplication operator $\M$ on $L^p(\Omega,X)$. The relevant stability properties are the following.
\begin{defn}Let $T$ be an operator on a Banach space $X$. Then $T$ is
\begin{itemize}
  \item[(i)] {\em uniformly stable} if $\left\|T^n\right\|\longrightarrow 0$ as $n\rightarrow\infty$.
  \item[(ii)] {\em strongly stable} if $\left\|T^nf\right\|\longrightarrow 0$ as $n\rightarrow\infty$ for all $f \in X$.
  \item[(iii)] {\em weakly stable} if $\psi\left(T^nf\right)\longrightarrow 0$ as $n\rightarrow\infty$ for all $f \in X$ and all $\psi\in X'$.
  \item[(iv)] {\em almost weakly stable} if $\psi(T^{n_k}f)\longrightarrow 0 \mbox{ as } k\rightarrow\infty$ for all $f \in X, \psi \in X'$, and sequences $(n_k)_{k\in\N} \subset \N$ of density 1.
\end{itemize}
\end{defn}
Recall that the {\em density of a sequence} $(n_k)_{k\in\N} \subset \N$ is
$$
d := \lim_{n\rightarrow\infty}{\frac{|\{ k\ :\ n_k < n \}|}{n}},
$$
if the limit exists and that a bounded linear operator $T$ on a Banach space is called {\em power bounded} if  $\sup_{n \in \mathbb{N}}\left\|T^n\right\| < \infty$.
By using methods analogous to those developed in Sections 1 -- 4 we can characterise these stability properties of a power bounded multiplication operator $\M$ through the pointwise operators $M(\cdot)$. 
\begin{thm}
Let $\M$ be a power-bounded multiplication operator on $L^p(\Omega,X)$ with $1<p<\infty$.
\begin{itemize}
  \item[(i)] Then $\M$ is uniformly stable if and only if $M(s)$ is uniformly stable for almost all $s \in \Omega$ and $\essup_{s\in\Omega}{\left\|M(s)^n\right\|} \rightarrow 0$ as $n\rightarrow\infty$, or, equivalently, that $\essup_{s\in\Omega}{\sr{M(s)}} < 1$.
  \item[(ii)] If $\Omega$ is $\sigma$-finite, then $\M$ is strongly stable if and only if $M(s)$ is strongly stable for almost all $s \in \Omega$.
  \item[(iii)] If the Banach space $X$ is separable, then $\M$ is weakly stable if the pointwise operators $M(s)$ are weakly stable almost everywhere.
  \item[(iv)] Let the measure space $(\Omega,\Sigma,\mu)$ be separable and let $X$ be a reflexive, separable Banach space. Then $\M$ is almost weakly stable if and only if, for each $\lambda\in\C$ with $|\lambda| = 1$, the rescaled pointwise operators $\lambda M(s)$ are Ces\`aro stable for almost all $s \in \Omega$.
\end{itemize}
\end{thm}
\bibliographystyle{siam}
\bibliography{retha}

\begin{thebibliography}{10}

\bibitem{abhn01}
{\sc W.~Arendt, C.~J.~K. Batty, M.~Hieber, and F.~Neubrander}, {\em
  Vector-valued {L}aplace transforms and {C}auchy problems}, vol.~96 of
  Monographs in Mathematics, Birkh\"auser/Springer Basel AG, Basel, second~ed.,
  2001.

\bibitem{ThoAr05}
{\sc W.~Arendt and S.~Thomaschewski}, {\em Local operators and forms},
  Positivity, 9 (2005), pp.~357--367.

\bibitem{BEP06}
{\sc A.~B{\'a}tkai, K.-J. Engel, J.~Pr{\"u}ss, and R.~Schnaubelt}, {\em
  Polynomial stability of operator semigroups}, Math. Nachr., 279 (2006),
  pp.~1425--1440.

\bibitem{CIZ09}
{\sc R.~Curtain, O.~V. Iftime, and H.~Zwart}, {\em System theoretic properties
  of a class of spatially invariant systems}, Automatica, 45 (2009), pp.~1619
  -- 1627.

\bibitem{DU77}
{\sc J.~Diestel and J.~J. Uhl, Jr.}, {\em Vector Measures}, American
  Mathematical Society, Providence, R.I., 1977.
\newblock With a foreword by B. J. Pettis, Mathematical Surveys, No. 15.

\bibitem{Don95}
{\sc R.~E. Edwards}, {\em Functional Analysis}, Dover Publications Inc., New
  York, 1995.
\newblock Theory and applications, Corrected reprint of the 1965 original.

\bibitem{eis10}
{\sc T.~Eisner}, {\em Stability of {O}perators and {O}perator {S}emigroups},
  vol.~209 of Operator Theory: Advances and Applications, Birkh\"auser Verlag,
  Basel, 2010.

\bibitem{en97}
{\sc K.-J. Engel}, {\em {Operator Matrices and Systems of Evolution
  Equations}}.
\newblock Manuscript, 1997.

\bibitem{enna00}
{\sc K.-J. Engel and R.~Nagel}, {\em One-{P}arameter {S}emigroups for {L}inear
  {E}volution {E}quations}, vol.~194 of Graduate Texts in Mathematics,
  Springer, New York, 2000.
\newblock With contributions by S. Brendle, M. Campiti, T. Hahn, G. Metafune,
  G. Nickel, D. Pallara, C. Perazzoli, A. Rhandi, S. Romanelli and R.
  Schnaubelt.

\bibitem{Gra97}
{\sc T.~Graser}, {\em Operator multipliers generating strongly continuous
  semigroups}, Semigroup Forum, 55 (1997), pp.~68--79.

\bibitem{HaWa96}
{\sc V.~Hardt and E.~Wagenf{\"u}hrer}, {\em Spectral properties of a
  multiplication operator}, Math. Nachr., 178 (1996), pp.~135--156.

\bibitem{Hol91}
{\sc A.~Holderrieth}, {\em Matrix multiplication operators generating one
  parameter semigroups}, Semigroup Forum, 42 (1991), pp.~155--166.

\bibitem{KurRyll_65}
{\sc K.~Kuratowski and C.~Ryll-Nardzewski}, {\em A general theorem on
  selectors}, Bull. Acad. Polon. Sci. S\'er. Sci. Math. Astronom. Phys., 13
  (1965), pp.~397--403.

\bibitem{MuNi11}
{\sc D.~Mugnolo and R.~Nittka}, {\em Properties of representations of operators
  acting between spaces of vector-valued functions}, Positivity, 15 (2011),
  pp.~135--154.

\bibitem{RudAna}
{\sc W.~Rudin}, {\em Real and Complex Analysis}, McGraw-Hill Book Co., New
  York, second~ed., 1974.
\newblock McGraw-Hill Series in Higher Mathematics.

\bibitem{Tho03}
{\sc S.~Thomaschewski}, {\em Form methods for autonomous and nonautonomous
  {C}auchy problems}, PhD-Thesis, Uni Ulm, 2003.

\bibitem{Neerven93}
{\sc J.~M. A.~M. van Neerven}, {\em Abstract multiplication semigroups}, Math.
  Z., 213 (1993), pp.~1--15.

\bibitem{Zwart}
{\sc H.~Zwart}, {\em Stability}, unpublished,  (2008).

\end{thebibliography}

\end{document}